\documentclass[12pt]{article}

\usepackage{graphicx}
\usepackage{ytableau}
\usepackage{microtype}
\usepackage{amsmath}
\usepackage{amsthm}
\usepackage{amssymb}
\usepackage{hyperref}
\usepackage{geometry}
\usepackage{epstopdf}

\newtheorem{theorem}{Theorem}[section]
\newtheorem{lemma}[theorem]{Lemma}

\newenvironment{definition}[1][Definition]{\begin{trivlist}
\item[\hskip \labelsep {\bfseries #1}]}{\end{trivlist}}

\newenvironment{remark}[1][Remark]{\begin{trivlist}
\item[\hskip \labelsep {\bfseries #1}]}{\end{trivlist}}

\newcommand{\hg}{\ensuremath{\!\:_0F_1}}
\newcommand{\hgq}{\ensuremath{\!\:_0\Phi_1}}

\newcommand\cminus{\mathbin{\raisebox{-\height}{$-$}}}
\newcommand\cplus{\mathbin{\raisebox{-\height}{$+$}}}

\usepackage{tikz}
\usetikzlibrary{decorations.pathmorphing}
\usetikzlibrary{arrows}
\usetikzlibrary{chains,fit,shapes}

\tikzset{snake it/.style={decorate, decoration=snake}}

\DeclareMathOperator{\maj}{maj}

\DeclareMathOperator{\tmaj}{tmaj}
\DeclareMathOperator{\SYT}{SYT}

\everymath{\displaystyle}

\title{Hook-length formula and applications to alternating permutations}

\author{Lucas Randazzo\thanks{\href{mailto:lucas.randazzo@u-pem.fr}{lucas.randazzo@u-pem.fr}.}}

\begin{document}

\maketitle

\begin{abstract}
In this paper, we take interest in finding applications for a hook-length formula recently proved in (Morales Pak Panova 2016). This formula can be applied to give a non trivial relation between alternating permutations and weighted Dyck paths. First, we give an alternative proof for this result using continued fractions, and then we apply a similar reasoning to the more general case of $k$-alternating permutations.
\end{abstract}

\section{Introduction}
The \emph{hook-length formula} for the number of standard Young tableaux of a Young diagram is a staple result in enumerative combinatorics, and as such has been widely studied. This formula has been discovered by Frame, Robinson and Thrall \cite{frame1954hook} in 1954, but only recently has a similar formula been proven in the more general case of skew tableaux by Naruse \cite{naruse2014schubert} in 2014. Two years later, Morales, Pak and Panova \cite{morales2016hook} proved a $q$-generalization of Naruse formula. They also applied their formula to a certain family of skew shapes, resulting in a non trivial relation between alternating permutations and weighted Dyck paths. The main goal of this paper is to explore different implications of this relation, while giving an alternative proof not involving the hook-length formula. Even though most continued fractions for generating functions are for natural series, as seen in Flajolet's theory \cite{flajolet1980combinatorial}, our proof instead relies on a continued fraction of the exponential series for Euler numbers.

In the next section, we introduce the main definitions and results that will be used in the following sections. We review the definitions of alternating permutations and Euler numbers. We then present Dyck paths and hypergeometric functions, then discuss their role in Flajolet's theory on continued fractions. Finally, we discuss about the hook-length formula and its applications. In the third section, we give a proof of Theorem \ref{th:toprove} without using the hook-length formula, then we prove a similar formula for alternating permutations of even length. In the last section, we discuss a generalization to $k$-alternating permutations.
\section{Definitions and preliminary results}
\subsection{Alternating permutations and Euler numbers}
In this article we focus our interest on alternating permutations, which are permutations that are defined from a condition involving their descent set. 
\begin{definition}
Let $\sigma \in \mathfrak{S}_n$, and $i \in \{1,\ldots,n-1\}$. We say that $i$ is a \emph{descent} of $\sigma$ if $\sigma_i > \sigma_{i+1}$. The \emph{descent set} of $\sigma$ is written $D(\sigma)$, and is the set if all descents of $\sigma$.
\end{definition}
\begin{definition}
Let $\sigma \in \mathfrak{S}_n$. We say that $\sigma$ is \emph{alternating} if $D(\sigma) = \{1,3,\ldots,2\left\lfloor{n}/{2}\right\rfloor-1\}$, \emph{i.e.} $\sigma_1~>~\sigma_2~<~\sigma_3~>~\cdots$. We denote by $\textit{Alt}_n$ the set of alternating permutations of $\mathfrak{S}_n$. Similarly, $\sigma$ is \emph{reverse alternating} if $D(\sigma) = \{2,4,\ldots,2\left\lceil{n}/{2}\right\rceil - 2\}$, \emph{i.e.} $\sigma_1~<~\sigma_2~>~\sigma_3~<~\cdots$. We denote by $\textit{RevAlt}_n$ the set of reverse alternating permutations of $\mathfrak{S}_n$.
\end{definition}
Alternating permutations in $\mathfrak{S}_n$ are enumerated by the \emph{Euler number} $E_n$. In his survey \cite{stanley2010survey}, Stanley gives some enumerative properties of alternating permutations. In particular, 
\begin{equation}
\label{eq:tansec}
\tan(z)=\sum_{n\geq 0}\frac{E_{2n+1}}{(2n+1)!}z^{2n+1},\hspace{1cm}\sec(z)=\sum_{n\geq 0}\frac{E_{2n}}{(2n)!}z^{2n}.
\end{equation}
Hence, we may also refer to $\{E_{2n+1}\}$ as the \emph{tangent numbers}, and $\{E_{2n}\}$ as the \emph{secant numbers}. Stanley also defines a natural $q$-analogue for the Euler numbers as follows
\[
E_n(q) := \sum_{\sigma \in \textit{RevAlt}_n}q^{\maj(\sigma^{-1})},\hspace{1cm}E_n^*(q) := \sum_{\sigma \in \textit{Alt}_n}q^{\maj(\sigma^{-1})},
\]
where $\maj(\sigma) := \sum_{i \in D(\sigma)}i$ is the \emph{major index} of $\sigma$. We notice that $E_n(1) = E_n^*(1) = E_n$, and more generally we have the relation

\begin{equation}
\label{eq:imp}
E^*_n(q)  =  q^{\binom{n}{2}}E_n(1/q).
\end{equation}

These refinements extend to $q$-analogues for the trigonometric functions $\tan$ and $\sec$. For $a \in \mathbb{N}$, let $(a)_n = \prod_{i=0}^{n-1}(a+i)$ be the Pochhammer symbol, and $(a;q)_n = \prod_{i=0}^{n-1}(1-aq^i)$ be its $q$-analogue. If
\[
\arraycolsep=1.4pt\def\arraystretch{2.2}
\begin{array}{rclcrcl}
\sin_q(z)&:=&\sum_{n \geq 0}(-1)^n\frac{u^{2n+1}}{(q;q)_{2n+1}},& &\cos_q(z)&:=&\sum_{n \geq 0}(-1)^n\frac{u^{2n}}{(q;q)_{2n}}, \\
\sin^*_q(z)&:=&\sum_{n \geq 0}(-1)^nq^{\binom{2n+1}{2}}\frac{u^{2n+1}}{(q;q)_{2n+1}},& &
\cos^*_q(z)&:=&\sum_{n \geq 0}(-1)^nq^{\binom{2n}{2}}\frac{u^{2n}}{(q;q)_{2n}}, \\
\end{array}
\]
then, from Theorem 2.1 in \cite{stanley2010survey}, we have
\[
\arraycolsep=1.4pt\def\arraystretch{2.2}
\begin{array}{rcccc}
\sum_{n \geq 0}\frac{u^{n}}{(q;q)_{n}}E_{n}(q)&=&\frac{\sin_q(z)}{\cos_q(z)} + \frac{1}{\cos_q(z)}&:=&\tan_q(z)+\sec_q(z),\\
\sum_{n \geq 0}\frac{u^{n}}{(q;q)_{n}}E^*_{n}(q)&=&\frac{\sin^*_q(z)}{\cos^*_q(z)} + \frac{1}{\cos^*_q(z)}&:=&\tan^*_q(z)+\sec^*_q(z).\\
\end{array}
\]
Also note the following identities, linking both variants of $\cos$ and $\sin$:
\[
\arraycolsep=1.4pt\def\arraystretch{2.2}
\begin{array}{rclcrcl}
\cos^*_q(x) & = & \cos_{1/q}(-x/q),& &
\sin^*_q(x) & = & \sin_{1/q}(-x/q).
\end{array}
\]
\begin{remark}
With (\ref{eq:imp}), it follows that $\tan_q^* = \tan_q$. This can also be proven using a simple bijection between $\textit{RevAlt}_n$ and $\textit{Alt}_n$ preserving $\maj(\sigma^{-1})$ when $n$ is odd: take $\sigma \in \textit{Alt}_n$, define $\sigma'$ as $\sigma'(i) = n-\sigma(n-i)$. Then $\sigma' \in \textit{RevAlt}_n$, and $\maj(\sigma^{-1}) = \maj(\sigma'^{-1})$.
\end{remark}
\subsection{Dyck paths}
Later in this article, we show the relation between alternating permutations and weighted Dyck paths. But beforehand it is important to understand how the generating function of Dyck paths can be written as a continued fraction.
\begin{definition}
A \emph{Dyck Path of length $2n$} is a path in the upper-right quarter plane whose steps are in $\{(1,1),(1,-1)\}$, begins at $(0,0)$ and ends at $(2n,0)$. Let $\mathcal{D}_n$ be the set of such paths.
\end{definition}
For convenience, we represent Dyck paths as the succession of heights it reaches rather than the steps it takes: let $(h_0,h_1,\ldots,h_{2n})\in\mathcal{D}_n$, so that $h_0 = h_{2n} = 0$, and for all $i$, $h_i \geq 0$ and $|h_i-h_{i+1}| = 1$. Thus, $h_i$ is the height the path reaches after $i$ steps. In the rest of this paper, we use the space-saving notation for continued fractions:
\[
\frac{a_0}{b_0} \cminus \frac{a_1}{b_1} \cminus \frac{a_2}{b_2} \cminus \cdots = 
\cfrac{a_0}{b_0-\cfrac{a_1}{b_1-\cfrac{a_2}{b_2 - \ddots}}}.
\]
The following lemma is a corollary of Flajolet's results on continued fractions in \cite{flajolet1980combinatorial}.
\begin{lemma}
\label{lem:fc}
Let $\mathbb{K}$ be a commutative field, and let $\omega : \mathbb{N} \rightarrow \mathbb{K}\setminus \{0\}$ be a weight function over the heights of a Dyck path. Then the formal series of weighted Dyck paths is a Stieltjes continued fraction:
\[
\arraycolsep=1.4pt\def\arraystretch{2.2}
\begin{array}{rcl}
\sum_{n\geq 0}\sum_{(h_0,h_1,\ldots,h_{2n}) \in \mathcal{D}_n}\prod_{i=0}^{2n} \omega(h_i)z^{2n} &=& \frac{\omega(0)}{1} \cminus \frac{\omega(0)\omega(1)z^2}{1} \cminus\frac{\omega(1)\omega(2)z^2}{1} \cminus \cdots\\
&=& \frac{1}{\omega(0)^{-1}} \cminus \frac{z^2}{\omega(1)^{-1}} \cminus\frac{z^2}{\omega(2)^{-1}} \cminus \cdots.\\
\end{array}
\]
\end{lemma}

\begin{remark}
The second equality is obtained by multiplying by $\omega(0)^{-1}$ on both sides of the first fraction, then by $\omega(1)^{-1}$ on the second, and so on.
\end{remark}
\begin{remark}
We mostly apply this lemma with $\mathbb{K} = \mathbb{R}$ or $\mathbb{K} = \mathbb{R}(X)$.
\end{remark}

\subsection{Hypergeometric functions}
It is useful to see continued fractions as quotient of \emph{hypergeometric functions}, since they satisfy the conditions of the following lemma.

\begin{lemma}
\label{lem:makefc}
For $i \geq 0$, let $f_i(z) = \sum_{n \geq 0} a_{i,n}z^n$ be a family of formal series. If for all $i$, there exists $k_i$ such that $ f_{i-1}~-~f_i~=~k_i~z~f_{i+1} $,
then
\[
\cfrac{f_1}{f_0} = \frac{1}{1} \cplus \frac{k_1z}{1} \cplus \frac{k_2z}{1} \cplus \cdots.
\]
\end{lemma}
Consider the \emph{hypergeometric function} $\hg$, defined as 
\[
\hg(a;z) = \sum_{n \geq 0} \frac{1}{n!(a)_n}z^n.
\]

Using Lemma \ref{lem:makefc}, these functions naturally lead to some interesting continued fractions.
\begin{lemma}
For all $a \in \mathbb{C}\setminus(\mathbb{Z}_{\leq 0})$,
\[
\frac{\hg(a+1;z)}{\hg(a;z)} = \frac{1}{a} \cplus \frac{z}{a+1} \cplus \frac{z}{a+2} \cplus \cdots.
\]
\end{lemma}
\begin{proof}
It results from the identity \[\hg(a-1;z) - \hg(a;z) = \frac{z}{a(a-1)}\hg(a+1;z),\] to which we apply Lemma \ref{lem:makefc} with $f_i = \hg(i;z)$ and $k_i=\frac{1}{i(i-1)}$.
\end{proof}
Since we have
\[
\arraycolsep=1.4pt\def\arraystretch{2.2}
\begin{array}{rcl}
\cos(z)&=&\hg(\frac{1}{2};-\frac{z^2}{4}),\\
\sin(z)&=&z\hg(\frac{3}{2};-\frac{z^2}{4}),
\end{array}
\]
we can write $\tan$ as a continued fraction
\begin{equation}
\label{eq:tan_cf}
\tan(z) = \frac{z\hg(\frac{3}{2};-\frac{z^2}{4})}{\hg(\frac{1}{2};-\frac{z^2}{4})} = \frac{z}{1} \cminus \frac{z^2}{3} \cminus \frac{z^2}{5} \cminus \cdots,
\end{equation}
which is also known as \emph{Lambert's continued fraction}.

The \emph{basic hypergeometric functions} $\hgq$ are $q$-analogues of hypergeometric functions, and are defined as
\[
\hgq(;a;q;z) = \sum_{n \geq 0} \frac{z^nq^{n(n-1)}}{(a;q)_n(q;q)_n}.
\]
\begin{lemma}
\label{lem:hgq_fc}
For all $a \in \mathbb{C}\setminus(\mathbb{Z}^-\cup\{0\})$,
\[
\frac{\hgq(;q^{a+1};q;qz)}{\hgq(;q^a;q;z)} = \frac{1}{1} \cplus \frac{\frac{1}{(1-q^a)(1-q^{a+1})}z}{1} \cplus \frac{\frac{q}{(1-q^{a+1})(1-q^{a+2})}z}{1} \cplus \cdots.
\]
\end{lemma}
\begin{proof}
For all $i \geq 0$, let $f_i =\ _0\Phi_1(;q^{a+i};q;zq^i)$. We have
\[
\arraycolsep=1.4pt\def\arraystretch{2.2}
\begin{array}{ll}
f_{i-1} - f_i = \hgq(;q^{a+i-1};q;zq^{i-1}) - \hgq(;q^{a+i};q;zq^i)\\
 = \sum_{n \geq 0}\left( \frac{q^{n(n-1)+n(i-1)}}{(q^{a+i-1};q)_n(q;q)_n} - \frac{q^{n(n-1)+ni}}{(q^{a+i};q)_n(q;q)_n} \right) z^n\\
 = \sum_{n \geq 0} \left( \frac{q^{n(n-1)+n(i-1)}(1-q^{a+i-1+n})}{(1-q^{a+i-1})(1-q^{a+i})} - \frac{q^{n(n-1)+ni}}{(1-q^{a+i})} \right)\frac{z^n}{(q^{a+i+1};q)_{n-1}(q;q)_n}\\
 = \sum_{n \geq 1} \frac{q^{n(n-1)+n(i-1)}(1-q^{a+i-1+n}) - q^{n(n-1)+ni}(1-q^{a+i-1})}{(1-q^{a+i-1})(1-q^{a+i})(1-q^n)}\frac{z^n}{(q^{a+i+1};q)_{n-1}(q;q)_{n-1}}\\
 = \frac{z}{(1-q^{a+i-1})(1-q^{a+i})} \sum_{n \geq 1} (q^{n(n-1)+n(i-1)})\frac{z^{n-1}}{(q^{a+i+1};q)_{n-1}(q;q)_{n-1}}\\
 = \frac{zq^{i-1}}{(1-q^{a+i-1})(1-q^{a+i})}  \sum_{n \geq 1} \frac{q^{(n-1)(i+1)}q^{(n-1)(n-2)}z^{n-1}}{(q^{a+i+1};q)_{n-1}(q;q)_{n-1}}\\
 =\frac{zq^{i-1}}{(1-q^{a+i-1})(1-q^{a+i})}f_{i+1}.\\
\end{array}
\]
The result follows from Lemma \ref{lem:makefc}.
\end{proof}
\subsection{Hook formula for skew shapes}
In this section, we present a Hook formula for skew shapes due to Morales, Pak and Panova in \cite{morales2016hook}, and show how it links alternating permutations with Dyck paths. Let us first give some useful definitions. Let $\lambda = ( \lambda_1, \ldots , \lambda_r )$ be an integer partition of length~$r$. The size of the partition is denoted by $|\lambda|$, and $\lambda'$ denotes the conjugate partition of $\lambda$. The \emph{Young diagram} associated with $\lambda$, denoted by $[\lambda]$, is a diagram made of squares, where the $i$-th row is composed of $\lambda_i$ squares. We use the french notation to represent them. The \emph{hook length} $h(i,j)$ of a square $(i,j)$ is the number of squares directly to the right and above the square in $[\lambda]$, including itself. A \emph{standard Young tableau} of shape $\lambda$ is an array $T$ of integers $1,\ldots,n$ of shape $\lambda$ strictly increasing in columns and strictly decreasing in rows. We note $\SYT(\lambda)$ the set of such tableaux. A \emph{descent} of a standard Young tableau $T$ is an index $i$ such that $i+1$ appears in a row above $i$. The \emph{major index} $\tmaj(T)$ is the sum over all the descents of $T$. Introduced by Frame, Robinson, and Thrall in \cite{frame1954hook}, the \emph{hook length formula} gives the number of standard Young tableaux of a given shape, and is stated as follows.

\begin{theorem}[Hook length formula \cite{frame1954hook}]
Let $\lambda$ be a partition of $n$. We have
\[
|\SYT(\lambda)| = \frac{n!}{\prod_{u \in [\lambda]}h(u)}.
\]
\end{theorem}

This can be extended to skew shapes. Let $\lambda$ and $\mu$ be two integer partitions. We write that $\mu \subset \lambda$ if, for all $i$, $\mu_i \leq \lambda_i$. In this case, we can define the \emph{skew shape} $\lambda / \mu$ as the shape obtained by removing $\mu$ from $\lambda$. Remark that the definition of standard Young tableaux can be easily extended to skew shapes. An \emph{excited diagram} of $\lambda / \mu$ is a subset of $[\lambda]$ of size $|\mu|$ obtained from $[\mu]$ by a sequence of excited moves:
\[
\ytableausetup{smalltableaux,centertableaux}
\ydiagram{2,2}*[*(gray)]{0,1} \hspace{0.2cm} \rightarrow \hspace{0.2cm} \ydiagram{2,2}*[*(gray)]{1+1,0}
\]
The move $(i,j) \rightarrow (i+1,j+1)$ is allowed only if cells $(i,j+1)$,$(i+1,j)$ and $(i+1,j+1)$ are unoccupied. The resulting set of excited diagrams of $\lambda / \mu$ is denoted $ \mathcal{E} (\lambda / \mu)$. Naruse in \cite{naruse2014schubert} obtained the following formula, generalizing the previous one.

\begin{theorem}[Naruse hook length formula \cite{naruse2014schubert}]
Let $\lambda, \mu$ be partitions, such that $\mu \subset \lambda$. We have
\[
|\SYT(\lambda / \mu)|= |\lambda / \mu|!\sum_{D \in \mathcal{E}(\lambda / \mu)}\prod_{u\in[\lambda]\setminus D}\frac{1}{h(u)}.
\]
\end{theorem}
Remark that if $\mu = \emptyset$, we obtain the original hook length formula.
Morales, Pak, and Panova in \cite{morales2016hook} proved a $q$-analogue of this formula.

\begin{theorem}[Morales, Pak, Panova, \cite{morales2016hook}]
\label{th:qhlf}
Let $\lambda, \mu$ be partitions, such that $\mu \subset \lambda$, and $n = |\lambda / \mu|$. We have
\[
\sum_{T \in SYT(\lambda / \mu)}q^{tmaj(T)} = (q;q)_n\sum_{D \in \mathcal{E}(\lambda / \mu)}\prod_{(i,j)\in[\lambda]\setminus D}\frac{q^{\lambda'_j-i}}{1-q^{h(i,j)}}.
\]

\end{theorem}

\begin{figure}
\center
\ytableausetup{nosmalltableaux}
\ytableaushort{6,13,\none 27, \none \none 45}*{1,2,3,4}*[*(gray)]{0,0,1,2}
\caption{The skew standard Young tableaux associated with the alternating permutation $(6,1,3,2,7,4,5)$.}
\label{fig:syt1}
\end{figure}

As means of example, they applied their result to the following skew shape: let $\delta_m = (m,m-1,\ldots,1)$ be an integer partition of the shape of a staircase of height $m$. Consider the border strip $\delta_{m+2} / \delta_m$. We can easily see that $\SYT(\delta_{m+2} / \delta_m)$ is in bijection with alternating permutations of length $2m+1$, as shown in Figure \ref{fig:syt1}. Moreover, $\mathcal{E}(\delta_{m+2} / \delta_m)$ is in bijection with Dyck paths of length $2m$. Moreover, in this case we have $\sum_{\sigma \in \textit{Alt}_{2m+1}}q^{\maj(\sigma^{-1})}~=~\sum_{T \in SYT(\delta_{m+2} / \delta_m)}q^{tmaj(T)}$. Hence, applying Theorem \ref{th:qhlf}, we obtain the following identity
\begin{theorem}[Morales, Pak, Panova, \cite{morales2016hook}]
\label{th:toprove}
\[
\sum_{d \in \mathcal{D}_n}\prod_{i=0}^{2n}\frac{q^{h_i}}{1-q^{2h_i+1}} = \frac{E_{2n+1}(q)}{(q;q)_{2n+1}}.
\]

\end{theorem}
Our goal in the next section is to give an alternative proof of this identity.
\section{Dyck paths and Euler numbers}
\subsection{Tangent numbers}

\begin{figure}
\centering
\includegraphics[scale=0.9]{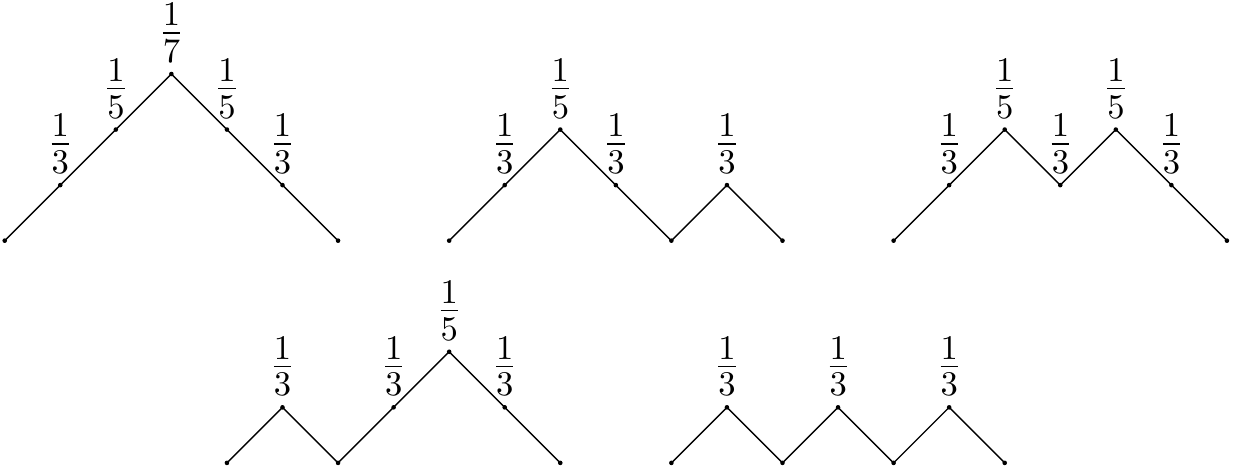}
\caption{The five weighted paths of $\mathcal{D}_3$, with vertices annotated with their weights.}
\label{fig:dyck}
\end{figure}

Before proving Theorem \ref{th:toprove}, we shall look at the corollary formula obtained with $q \rightarrow 1$:
\[
\sum_{(h_0,h_1,\ldots,h_{2n})  \in \mathcal{D}_n}\frac{(2n+1)!}{\prod_{i=0}^{2n}\left(2h_i+1\right)} = E_{2n+1}.
\]
The left hand side of this equation can be seen as the sum over all weighted Dyck paths of length $2n$. For instance, if $n=3$, 
we have all the weighted paths as shown in Figure \ref{fig:dyck}, and summing them up 
we get \[7!\cdot\left( \frac{1}{3^3} + \frac{1}{3^3\cdot5}+\frac{1}{3^3\cdot5}+\frac{1}{3^3\cdot5^2}+\frac{1}{3^2\cdot5^2\cdot7}\right) = 272 = E_7.\] Even though the weight of each path is a rational number, they non-trivially add up to an integer. This remark oriented our research towards a more direct approach to prove this formula.
\begin{proof}[Proof of Theorem \ref{th:toprove}]
Let $u(q;z) = \sum_{n \geq 0}\sum_{d \in \mathcal{D}_n}\prod_{i=0}^{2n}\frac{q^{h_i}}{1-q^{2h_i+1}}z^{2n}$. Using Lemma \ref{lem:fc} with $\omega(x)~=~\frac{q^x}{1-q^{2x+1}}$, we have
\[
u(q;z) = \frac{\frac{1}{1-q}}{1} \cminus \frac{\frac{q}{(1-q)(1-q^3)}z^2}{1} \cminus \frac{\frac{q^3}{(1-q^3)(1-q^5)}z^2}{1} \cminus \cdots .
\]
Using Lemma \ref{lem:hgq_fc}, with $z = \frac{-z^2q^a}{4}$, $a = \frac{1}{2}$, and $q = q^2$, we have
\[
\frac{1}{1-q}\frac{_0\Phi_1(;q^3;q^2;-z^2q^3)}{_0\Phi_1(;q;q^2;-z^2q)} = u(q;z).
\]
Remark that
\[
\frac{z\hgq(;q^3;q^2;-z^2q^3)}{1-q} = \sin^*_q(z),
\]
and
\[
_0\Phi_1(;q;q^2:-z^2q) = \cos^*_q(z).
\]
Hence,
\[
u(q;z) = \frac{\tan^*_q(z)}{z} = \frac{\tan_q(z)}{z}.
\]
\end{proof}
\begin{remark}
The proof uses a continued fraction for $tan_q$, that is
\[
\tan_q(z) = \frac{z}{1-q} \cminus \frac{z^2}{\frac{1-q^3}{q}} \cminus \frac{z^2}{\frac{1-q^5}{q^2}} \cminus \cdots,
\]
which gives back Lambert's continued fraction (\ref{eq:tan_cf}) for $\tan$ by taking $z = (1-q)z$ and $q \rightarrow 1$.
\end{remark}
Similarly, as we define $u_h(q;z)$ to be the generating function for $h$-shifted Dyck paths, we can prove that, for $h \geq 0$,
\begin{equation}
\label{eq:sdf}
u_h(q;z) = \frac{q^h}{1-q^{2h+1}}\frac{_0\Phi_1(;q^{2h+3};q^2;-z^2q^{2h+3})}{_0\Phi_1(;q^{2h+1};q^2:-z^2q^{2h+1})}.
\end{equation}

\subsection{Secant numbers}

We can prove a similar formula for $q$-Secant Numbers. Let us consider the skew shape $\delta^*_{n+2} / \delta_n$, where $\delta^*_n$ is the shape $\delta_n$ without its topmost square. This way, the hooks of the first column decrease by one. $\SYT(\delta^*_{n+2} / \delta_n)$ is in bijection with $\textit{Alt}_{2n}$, and $\mathcal{E}(\delta^*_{n+2} / \delta_n)$ is in bijection with Dyck paths of length $2n$ for which we remove the first step. The weight for the vertices of those paths is still $\frac{q^h}{1-q^{2h+1}}$ at height $h$, except for the first ascent of the path, for which it becomes $\frac{q^{h}}{1-q^{2h}}$, and we ignore the first step. We do the same for the conjugate skew shape. With Theorem \ref{th:qhlf}, it translates into the following equalities
\begin{theorem}

\[
\arraycolsep=1.4pt\def\arraystretch{2.2}
\begin{array}{rcl}
\sum_{d \in \mathcal{D}_n}\prod_{i=1}^{2n}\frac{q^{h_i}}{1-q^{2h_i+1-\delta_{h_i,i}}} &=& \frac{E^*_{2n}(q)}{(q;q)_{2n}},\\
\sum_{d \in \mathcal{D}_n}\prod_{i=1}^{2n}\frac{q^{h_i-\delta_{h_i,i}}}{1-q^{2h_i+1-\delta_{h_i,i}}} &=& \frac{E_{2n}(q)}{(q;q)_{2n}},\\
\end{array}
\]
where $\delta_{i,j}$ is the Kronecker delta.
\end{theorem}

\begin{proof}

\begin{figure}
\centering
\includegraphics[scale=0.9]{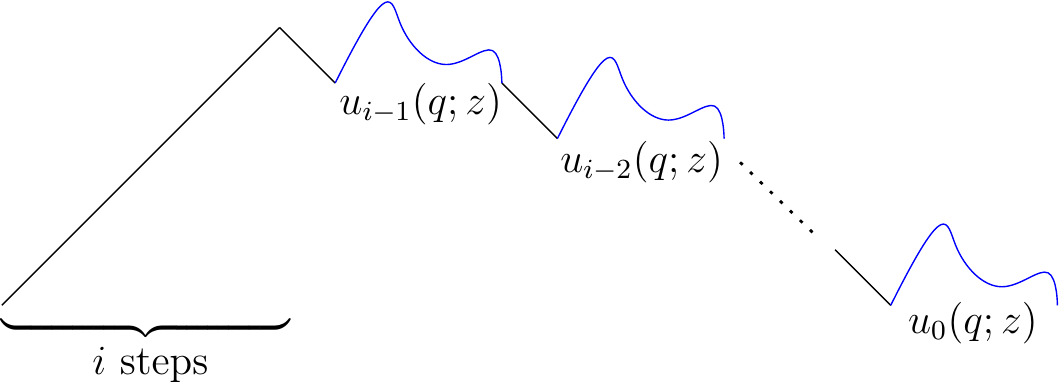}
\caption{An element of $\mathcal{D}_n$ starting with exactly $i$ ascending steps. The weight of the blue sub-paths is indicated beneath them.}
\label{fig:decomp}
\end{figure}

Both equalities are equivalent due to (\ref{eq:imp}), so we only need to prove the first one. Let $v(q;z)~=~\sum_{n \geq 1}\sum_{d \in \mathcal{D}_n}\prod_{i=1}^{2n}\frac{q^{h_i}}{1-q^{2h_i+1-\delta_{h_i,i}}}z^{2n}$ be the weighted generating function for Dyck paths without their first step. We decompose those paths as follows: the path begins with an ascent up to height $i$, in which each vertex has weight $\frac{q^{h}}{1-q^{2h}}$ at height $h$. It is followed $i$ times by a down step and a shifted weighted Dyck path, counted by $u_j(q;z)$. This decompositon is illustrated in Figure (\ref{fig:decomp}). Hence we have
\[
\arraycolsep=1.4pt\def\arraystretch{2.2}
\begin{array}{rcl}
v(q;z)&=&\sum_{i \geq 1} \left(\prod_{j=1}^{i} \frac{zq^j}{1-q^{2j}} \right) \left(\prod_{j=0}^{i-1}zu_j(q;z)\right)\\
&=&\sum_{i \geq 1}\frac{q^{i(i+1)/2}}{(q^2;q^2)_i} \left(\prod_{j=0}^{i-1}u_j(q;z)\right)z^{2i}.\\
\end{array}
\]
With (\ref{eq:sdf}), we have
\[
\arraycolsep=1.4pt\def\arraystretch{2.2}
\begin{array}{rcl}
v(q,z)&=& \sum_{i \geq 1}\frac{q^{i(i+1)/2}}{(q^2;q^2)_i} \left(\prod_{j=0}^{i-1}\frac{q^j}{1-q^{2j+1}}\frac{_0\Phi_1(;q^{2j+3};q^2;-z^2q^{2j+3})}{_0\Phi_1(;q^{2j+1};q^2:-z^2q^{2j+1})}  \right) z^{2i}\\
&=& \sum_{i \geq 1}\frac{q^{i^2}}{(q^2;q^2)_i(q;q^2)_i}\frac{_0\Phi_1(;q^{2i+1};q^2;-z^2q^{2i+1})}{_0\Phi_1(;q;q^2;-z^2q)}z^{2i}\\
&=& \sec^*_q(z)\left(\sum_{i \geq 0}\frac{q^{i^2}}{(q^2;q^2)_i(q;q^2)_i}\ _0\Phi_1(;q^{2i+1};q^2;-z^2q^{2i+1})z^{2i} \right) -1.\\
\end{array}
\]
Moreover, we have
\[
\arraycolsep=1.4pt\def\arraystretch{2.2}
\begin{array}{cl}
&\hspace{-2.2cm}\sum_{i \geq 0}\frac{q^{i^2}}{(q^2;q^2)_i(q;q^2)_i}\ _0\Phi_1(;q^{2i+1};q^2;-z^2q^{2i+1})z^{2i}\\
=&\sum_{i,j \geq 0}\frac{q^{i^2}}{(q^2;q^2)_i(q;q^2)_i}\frac{(-z^2q^{2i+1})^jq^{2j(j+1)}}{(q^2;q^2)_j(q^{2i+1};q^2)_j}z^{2i}\\
=&\sum_{i,j \geq 0}\frac{(-1)^jz^{2i+2j}q^{i^2+2ij+2j^2-j}}{(q^2;q^2)_i(q^2;q^2)_j(q;q^2)_{i+j}}\\
=&\sum_{n \geq 0}\frac{z^{2n}q^{n^2}}{(q;q^2)_n}\sum_{k = 0}^{n}\frac{(-1)^kq^{k(k-1)}}{(q^2;q^2)_k(q^2;q^2)_{n-k}}\\
=&\sum_{n \geq 0}\frac{z^{2n}q^{n^2}}{(q;q)_{2n}}\sum_{k = 0}^{n}\binom{n}{k}_{q^2}(-1)^kq^{k(k-1)}.
\end{array}
\]
To conclude, we need to use the \emph{Cauchy binomial theorem}. For all $n \in \mathbb{N}$, $k \leq n $, let $\binom{n}{k}_q~=~\frac{(q;q)_n}{(q;q)_k(q;q)_{n-k}}$ be the $q$-binomial coefficients.
\begin{theorem}[Cauchy binomial Theorem]
\label{th:binom}
\[
(a;q)_n = \sum_{k=0}^n\binom{n}{k}_q(-a)^kq^{\binom{k}{2}}.
\]
\end{theorem}
We can then simplify the previous expression, since we have
\[
\sum_{k = 0}^{n}\binom{n}{k}_{q^2}(-1)^kq^{k(k-1)} = (1;q^2)_n = \left\lbrace
\begin{array}{ll}
1&\text{if }n=0,\\
0&\text{otherwise.}
\end{array}\right.
\]
It follows that
\[
v(q,z) = \sec^*_q(z) - 1.
\]
\end{proof}
\section{Alternating $k$-permutations}

\begin{figure}
\center
\ytableausetup{nosmalltableaux}
\ytableaushort{267,\none \none 458, \none \none \none \none 13}*{3,5,6}*[*(gray)]{0,2,4}
\caption{The standard Young tableau associated with the alternating $3$-permutation $(2,6,7,4,5,8,1,3)$.}
\label{fig:shapes}
\end{figure}

In this section we try to generalise our analysis to alternating $k$-permutations.
\begin{definition}
Let $\sigma \in \mathfrak{S}_n$. We say that $\sigma$ is an \emph{alternating $k$-permutation} if \[D(\sigma) = \{ k, 2k, \ldots, (\left\lceil n/k \right\rceil~-~1~)~k~\}.\] We call $\mathcal{A}_k(n)$ the set of alternating $k$-permutations of $\mathfrak{S}_n$.
\end{definition}
When $k = 2$, we recover ordinary alternating permutations, since $\mathcal{A}_2(n) = \textit{RevAlt}_n$. Similarly to alternating permutations, we can put in bijection elements of $\mathcal{A}_k(n)$ with standard Young tableaux of a certain skew shape, like in Figure \ref{fig:shapes}.
Let $\sigma \in \mathcal{A}_k(nk+r)$ with $0 < r \leq k$, we call $T(\sigma)$ the associated standard Young tableau of $\sigma$, so that 
\[
maj(\sigma^{-1}) = \tmaj(T(\sigma)).
\]
Just like for alternating permutations, we can define the following generating function
\[
f_{nk+r}(q) = \sum_{\sigma \in \mathcal{A}_k(nk+r)} q^{maj(\sigma^{-1})}.
\]
and use Morales, Pak and Panova's hook formula, which gives us
\[
\frac{f_{nk+r}(q)}{(q;q)_{nk+r}} = \sum_{d \in \mathcal{D}_{nk+r,k}}\prod_{i=0}^{2n}\frac{q^{h_i}}{1-q^{h_i+\lfloor\frac{h_i}{k-1}\rfloor+1}},
\]
where $\mathcal{D}_{nk+r,k}$ 
is the set of $k$-Dyck paths of length $nk+r$, which is a path in the upper-right quarter plane whose steps are in $\{(1,k-1),(1,-1)\}$, begins at $(0,0)$ and ends at $(nk+r,k-r)$.
Instead, we look at Kur\c sung\"oz and Yee's work in \cite{kurcsungoz2011alternating} studying $f_{nk+r}(q)$. They found that, for $k \geq 2$ and $0 < r \leq k$,
\begin{equation}
F_r := \sum_{n \geq 0} \frac{f_{nk+r}(q)t^{nk+r}}{(q;q)_{nk+r}} = \frac{\displaystyle \sum_{n \geq 0} \frac{(-1)^nt^{nk+r}}{(q;q)_{nk+r}}}{\displaystyle \sum_{n \geq 0} \frac{(-1)^nt^{nk}}{(q;q)_{nk}}}.
\label{eq:KY}
\end{equation}
We give another proof of this result using a set of differential equations satisfied by this generating function. 
Take $\sigma \in \mathcal{A}_n$ and remove its maximum element to obtain $\sigma'$. Then we have $maj(\sigma^{-1}) = maj(\sigma'^{-1})$. Hence we have
\[
\left\lbrace\begin{array}{l}
F'_r = F_{k-1}F_{r} + F_{r-1} \text{ for }r \geq 2\text{, and}\\
F'_1 = F_{k-1}F_1 + 1.
\end{array}\right.
\]
We can then verify that (\ref{eq:KY}) is a solution to this set of equations. If we take $k = 2$, we find back the derivatives of $F_1 = \tan$ and $F_2 = \sec-1$, that is
\[
\begin{array}{l}
\sec' = \tan.\sec\\
\tan' = \tan^2 + 1.
\end{array}
\]

\section{Acknowledgements}
We thank Matthieu Josuat-Verg\`es for his invaluable help and support. This research did not receive any specific grant from funding agencies in the public, commercial, or not-for-profit sectors.

\bibliographystyle{abbrv}
\bibliography{arxiv_sub.bib}

\end{document}